\documentclass[11pt]{article}

\usepackage{tikz}
\usetikzlibrary{calc}
\tikzstyle{vertex}=[circle, draw, inner sep=0pt, minimum size=8pt]
\newcommand{\vertex}{\node[vertex]}

\usepackage{enumerate}
\usepackage{latexsym} 
\usepackage{theorem}
\usepackage{graphics}
\usepackage{graphicx}
\usepackage{amsmath}
\usepackage{amssymb}  
\usepackage[english]{babel} 
\usepackage{comment}
\usepackage{color}
\usepackage{tikz}
\usetikzlibrary{arrows}
\usepackage{ifpdf}

\usepackage{float}
\usepackage{amsfonts}
\usepackage{setspace}
\usepackage{fullpage}
\usepackage{enumitem}
\usepackage{hyperref}
\usepackage{microtype}
\usepackage{enumerate}

\newtheorem{lemma}{Lemma}[section]
\newtheorem{theorem}[lemma]{Theorem}

\newtheorem{corollary}[lemma]{Corollary}

\newtheorem{conjecture}[lemma]{Conjecture}

\tikzstyle{vertex}=[circle, draw, inner sep=0pt, minimum size=10pt]

\newcounter{claim}

\newenvironment{proof}[1][]%
 {\noindent {\setcounter{claim}{0}\sc proof ---
   }{#1}{}}{\hfill$\Box$\vspace{2ex}} 

\newenvironment{claim}[1][]%
{\refstepcounter{claim}\vspace{1ex}\noindent{(\it\arabic{claim}){#1}{}}\it}{\vspace{1ex}}

\newenvironment{proofclaim}[1][]%
	{\noindent {}{#1}{}}{ This proves~(\arabic{claim}).\vspace{2ex}}

\title{Vizing's and Shannon's Theorems for defective edge colouring}

\author{Pierre Aboulker$^1$, Guillaume Aubian$^{1,2}$, Chien-Chung Huang$^{1}$. \\
\small ($1$) DIENS, \'Ecole normale sup\'erieure, CNRS, PSL University, Paris, France.  \\
\small ($2$) Université de Paris, CNRS, IRIF, F-75006, Paris, France.\\
 \small{Mails: \tt pierreaboulker@gmail.com, guillaume.aubian@gmail.com,  Chien-Chung.Huang@ens.fr}}

\begin{document}

\maketitle
\begin{abstract}
We call a multigraph $(k,d)$-edge colourable if its edge set can be partitioned into  $k$ subgraphs of maximum degree at most $d$ and denote as $\chi'_{d}(G)$ the minimum $k$ such that $G$ is $(k,d)$-edge colourable.  
We prove that for every odd integer $d$, every multigraph $G$ with maximum degree $\Delta$ is  $(\lceil \frac{3\Delta - 1}{3d - 1} \rceil, d)$-edge colourable and that this bound is attained for all values of $\Delta$ and $d$. 
An easy consequence of Vizing Theorem is that $\chi'_{d}(G) \in \{ \lceil \frac{\Delta}{d} \rceil, \lceil \frac{\Delta+1}{d} \rceil \}$. We characterize the values of $d$ and $\Delta$ for which it is NP-complete to compute $\chi'_d(G)$. These results generalize  classic results on the chromatic index of a graph by Shannon, Holyer, Leven and Galil and extend a result of Amini, Esperet and ven den Heuvel.
\end{abstract}

\section{Introduction}
\emph{Graphs} in this paper are finite, undirected, and without loops, but may have multiple edges.  A graph is \textit{simple} if it has no multiple edge.
Let $G$ be a graph. 
We denote by $\Delta(G)$ the maximum degree of $G$. 
An \emph{edge colouring of $G$ with defect $d$} is a colouring of its edge set in such a way that each vertex is incident with at most $d$ edges of the same colour. We say that $G$ is \emph{$k$-edge colourable with defect $d$}, or simply \emph{$(k,d)$-edge colourable} if $G$ admits an edge colouring with defect $d$ using (at most) $k$ colours. In other words, the edge set can be partitioned into at most $k$ subgraphs of maximum degree $d$. 
 The \emph{$d$-defective chromatic index} of $G$ is the minimum $k$ such that $G$ is $(k,d)$-edge colourable and is denoted by $\chi^{'}_{d}(G)$. 
So $\chi'_1(G)$ is the usual chromatic index. 

This notion is called \textit{frugal edge colouring} in~\cite{amini:inria-00144318} and \textit{improper edge colouring} in~\cite{HILTON2001253}. Our vocabulary follows the one existing for the  analogue concept of defective vertex colouring, a now well established notion. See~\cite{W18} for a nice dynamic survey on defective vertex colouring. 



Our first result is the following.

\begin{theorem}\label{thm:main_theorem}
Let $d, \Delta \geq 1$ and let $G$ a graph with maximum degree $\Delta$. 
If $d$ is even, then $\chi'_{d}(G) = \lceil \frac{\Delta}{d} \rceil$, and if $d$ is odd, then $\chi'_{d}(G) \leq \lceil \frac{3\Delta - 1}{3d - 1} \rceil$. Moreover, these bounds are tight for every value  of $\Delta$ and $d$, see Lemma~\ref{lem:shannon_tight}.
\end{theorem}

The case $d=1$ corresponds to the classic result of Shannon~\cite{S49} on chromatic index  stating that for every graph $G$, $\chi'_1(G) \leq \lfloor \frac{3\Delta(G)}{2} \rfloor$ (observe that $\lceil \frac{3\Delta - 1}{2} \rceil = \lfloor \frac{3\Delta}{2} \rfloor$ whenever $\Delta \geq 1$). 
When $d$ is even, the result is almost trivial in our context (see Theorem~\ref{thm:deven}, and was already known in the more general context of list colouring~\cite{HILTON2001253, amini:inria-00144318}. 
When $d$ is odd, a proof that $\chi'_d(G) \leq \lceil \frac{3\Delta}{3d - 1} \rceil$ in the context of list coloring is announced in~\cite{amini:inria-00144318}, but seems to contain a flaw and actually holds only in the case where $\Delta$ is divisible by $3k-1$. See section~\ref{sec:further} for more on the list colouring context. 


\medskip

Vizing's celebrated theorem on edge colouring~\cite{V64} states that for every simple graph $G$, $\chi'_1(G) \in \{\Delta(G), \Delta(G) +1\}$, and Holyer~\cite{H81}, and Leven and Galil~\cite{LZ83} proved that deciding if $\chi'_1(G) =\Delta(G)$ is NP-complete even restricted to $d$-regular simple graphs as soon as $d \geq 3$. 
We generalize both results by proving that for every simple graph $G$, $\chi'_{d}(G) \in \{ \lceil \frac{\Delta}{d} \rceil, \lceil \frac{\Delta+1}{d} \rceil \}$ (which is easily implied by Vizing's Theorem) and we characterize the values of $\Delta$ and $d$ for which the problem is NP-complete. More precisely, we prove that, given a $\Delta$-regular simple graph, it is NP-complete to decide if $\chi'_d(G) = \lceil \frac{\Delta}{d} \rceil$ if and only if $d$ is odd and $\Delta = kd$ for some $k \geq 3$. See Theorem~\ref{thm:NP}. 
\medskip

We give some definitions and preliminary results in Section~\ref{sec:intro}. We prove the generalization of Shannon's Theorem in Section~\ref{sec:multi} and the proof of the generalization of Vizing's Theorem in Section~\ref{sec:simple}. Finally, in Section~\ref{sec:further}, we propose as  conjectures a generalisation of Theorem~\ref{thm:main_theorem} for list coloring and a generalisations of the Goldberg-Seymour Conjecture.
\medskip


\section{Definitions and preliminaries} \label{sec:intro}
Let $G$ be a graph. 
The \emph{size} of $G$ is its number of vertices. It is \textit{regular} if there is an integer $k$ such that every vertex of $G$ has degree $k$. In this case we can also say it is \textit{$k$-regular}. 
We say that $G$ is \emph{$k$-edge-connected} if it remains connected whenever (strictly) fewer than $k$ edges are removed. 
If $u, v \in V(G)$, we denote as $G + uv$ the graph $(V(G), E(G) \cup \{ uv \})$ (recall that in this paper graphs can have multiple edges, so if there is already an edge between $u$ and $v$, another one is added). Similarly, $G - uv = (V(G), E(G) \setminus \{ uv \})$.

The following gives a trivial lower bound on the $d$-defective chromatic index that turns out to be tight whenever $d$ is even (see Theorem~\ref{thm:deven}). 

\begin{lemma}\label{lem:trivialbound}
For every graph $G$, $\chi'_{d}(G) \geq \lceil \frac{\Delta(G)}{d} \rceil$.
\end{lemma}

\begin{proof}
A least $\lceil \frac{\Delta(G)}{d} \rceil$ colours are needed to colour the edges incident to a vertex of degree $\Delta(G)$.
\end{proof}


\begin{lemma}\label{lem:chi_increasing}
Let $k, d, \Delta$ be integers. 
If every $(\Delta + 1)$-regular graph is $(k,d)$-edge colourable, then every $\Delta$-regular graph is also  $(k,d)$-edge colourable. 
\end{lemma}

\begin{proof}
Let $G$ be a $\Delta$-regular graph. 
Take two disjoint copies $G'$ and $G''$ of $G$ and add an edge between each vertex $v \in V(G')$ and its copy in $G''$. The obtained graph $H$ is $(\Delta+1)$-regular and contains $G$ as a subgraph, so $\chi'_d(G) \leq \chi'_d(H) \leq k$. 
\end{proof}


\subsubsection*{Factors in graphs}

A \emph{$k$-factor} of $G$ is a $k$-regular spanning subgraph of $G$. 
We sometimes consider a $k$-factor $F$ as its edge set $E(F)$. 
We recall this theorem from Petersen \cite{P91}, one of the very first fundamental results in graph theory:

\begin{theorem}\cite{P91}\label{thm:petersen}
Let $\Delta$ be an even integer. A $\Delta$-regular graph admits a $k$-factor for every even integer $k \leq \Delta$. 
\end{theorem}

An Eulerian cycle of a graph $G$ is a cycle that uses every edge of $G$. It is a well-known fact that a graph admits an Eulerian cycle if and only if it is connected and all its vertices have even degree. 
The next two lemmas use this fact to prove the existence of factors. This idea was already used by Petersen to prove his theorem. 

\begin{lemma}\label{lem:delta_even_kV_even}
Let $G$ be a connected $2k$-regular graph with an even number of edges. Then the edges of $G$ can be partitioned into two $k$-factors. 
\end{lemma}

\begin{proof}
We number the edges $e_1, e_2, \dots, e_{2t}$ of $G$ along an Eulerian cycle $C$ and we let $A= \{e_1, e_3, \dots, e_{2t-1}\}$ and $B = \{e_2, e_4, \dots, e_{2t}\}$. 
Since consecutive edges of $C$ are numbered with different parities and its first and last edges  have distinct parities, $A$ and $B$ are both $k$-regular. 
\end{proof}

\begin{lemma}\label{lem:delta_even_kV_odd}
Let $G$ be a connected $2k$-regular graph with an odd number of edges, and let $e \in E(G)$. There exist two graphs $G_{A} = (V,A)$ and $G_{B} = (V,B)$ such that $E(G) = A \cup B \cup \{e\}$, $\Delta(G_{A}) \leq k$ and $\Delta(G_{B}) \leq k$.
\end{lemma}

\begin{proof}
The proof is the same as for the previous Lemma, except that we do not assign the last edge of the Eulerian cycle, and we choose $e$ to be this last edge. 
\end{proof}




The next theorem roughly says that, in a $\Delta$-regular graph, one can find a $k$-factor  as soon as $k$ is even and is relatively small compared to $\Delta$. It was first proved in \cite{K84}. See also Theorem 3.10 $(v)$ in \cite{AK11}). The version stated here is a simplified version of the original theorem.

\begin{theorem}\cite{K84}\label{lem:factor2/3}\label{lem:delta_odd_k_even_factor_including}
Let $\Delta$ be an odd integer and $G$ a $2$-edge connected $\Delta$-regular graph.  Let $e \in E(G)$. Let $k$ be an even integer with $k \leq \frac{2\Delta}{3}$. Then $G$ has a $k$-factor containing $e$. 
\end{theorem}



\subsubsection*{Shannon graphs}

Given an integer $k$, the \emph{Shannon graph} $Sh(k)$ is the graph made of three vertices connected by $\lfloor \frac{k}{2} \rfloor$, $\lfloor \frac{k}{2} \rfloor$ and $\lceil \frac{k}{2} \rceil$ edges respectively. See Figure~\ref{fig:shanon}. Observe that 
\begin{itemize}
    \item $\Delta(Sh(k))=k$,
    \item  when $k$ is even, $Sh(k)$ is $k$-regular and has $\frac{3k}{2}$ edges and,
    \item when $k$ is odd, $Sh(k)$ has two vertices of degree $k$ and one vertex of degree $k-1$ and has $\frac{3k-1}{2}$ edges. 
\end{itemize}

    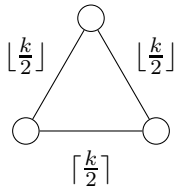
\begin{figure}[H]
    \centering
    \begin{tikzpicture}
        \begin{scope}
        \vertex (1) at (0,1) {};
        \vertex (2) at (-0.866,-0.5) {};
        \vertex (3) at (0.866,-0.5) {};
        \node (4) at (-0.866,0.5) {$\lfloor \frac{k}{2} \rfloor$};
        \node (5) at (0.866,0.5) {$\lfloor \frac{k}{2} \rfloor$};
        \node (6) at (0,-1) {$\lceil \frac{k}{2} \rceil$};
        \draw (1) -- (2);
        \draw (2) -- (3);
        \draw (3) -- (1);
        \end{scope}
    \end{tikzpicture}
    \caption{The Shannon graph $Sh(k)$} \label{fig:shanon}
    \end{figure}

\begin{lemma}\label{lem:shannon_tight}
Let $k,d \in \mathbb N^*$ with $d$ odd. Then $\chi'_d(Sh(k)) = \lceil\frac{3k-1}{3d-1}\rceil$. 
\end{lemma}

\begin{proof}
Consider an ordering $(e_i)_{1 \leq i \leq |E(Sh(k))|}$ of the edges of $Sh(k)$ such that for any $1 \leq i \leq |E(Sh(k))| - 2$, $e_i$, $e_{i+1}$ and $e_{i+2}$ forms a triangle. 
Such an ordering can be obtained by setting $e_{1}$ to be any edge with both extremities of degree $k$ and then setting, for $i=2, \dots, |E(Sh(k))| - 1$, $e_{i+1}$ to be any unnumbered edge coming right after $e_{i}$ in clockwise order. 
The following statement is easily proven using induction:
\emph{
    For every odd integer $\ell$ such that $3 \leq \ell \leq \frac{2|E(Sh(k)|-1}{3}$, every contiguous subsequence of $(e_{i})_{1 \leq i \leq |E(Sh(k))|}$ of length $\frac{3\ell - 1}{2}$ induces a graph of maximum degree $\ell$.
}

Thus, colouring the first $\frac{3d-1}{2}$ edges of $(e_i)_{1 \leq i \leq |E(Sh(k))|}$ in one colour, the following $\frac{3d-1}{2}$ in a second colour and so on, yields a colouring with at most $\lceil \frac{|E(Sh(k))|}{\frac{3d-1}{2}} \rceil$ colours such that each colour class induces a subgraph with maximum degree at most $d$, and each colour class except at most one has $\frac{3d-1}{2}$ edges. Since every subgraph of $Sh(k)$ with maximum degree $d$ (recall that $d$ is odd) has at most $\frac{3d-1}{2}$ edges, this colouring is an optimal d-defective edge colouring and thus: 
\[ \chi'_d(Sh(k)) = \Bigl\lceil \frac{|E(Sh(k))|}{\frac{3d-1}{2}} \Bigr\rceil = 
  \begin{cases}
     \lceil \frac{3k}{3d-1} \rceil = \lceil \frac{3k-1}{3d-1} \rceil   & \text{ if } k \text{ is even} \\
     \lceil \frac{3k-1}{3d-1} \rceil & \text{ if } k \text{ is odd}
  \end{cases}
\]
\end{proof}


\section{Generalization of Shannon Theorem} \label{sec:multi}


The goal of this section is to prove Theorem~\ref{thm:main_theorem}. The case where $d$ is even was already known, we give the proof anyway for completeness. 
\begin{theorem}\cite{HILTON2001253, amini:inria-00144318}\label{thm:deven}
Let $d, \Delta \in \mathbb N^*$ with $d$ even. For every $\Delta$-regular graph $G$, $\chi^{'}_{d}(G) = \lceil \frac{\Delta}{d} \rceil$.
\end{theorem}

\begin{proof}
If $\Delta$ is even, then $\chi'_d(G) \leq \lceil \frac{\Delta(G)}{d} \rceil$ by Theorem~\ref{thm:petersen}. 
If $\Delta$ is odd, then 
$\lceil \frac{\Delta+1}{d} \rceil = \lceil \frac{\Delta}{d} \rceil $  
and by Lemma~\ref{lem:chi_increasing}, 
$\chi'_d(G) \leq  \lceil \frac{\Delta}{d} \rceil$.
Equality holds in both cases by Lemma~\ref{lem:trivialbound}.  
\end{proof}

\begin{theorem}
Let $d, \Delta \in \mathbb N^*$ with $d$ odd. 
Then every graph $G$ with maximum degree $\Delta$ is $(\lceil \frac{3\Delta - 1}{3d - 1} \rceil, d)$-edge colourable. 
\end{theorem}

\begin{proof}
If $d = 1$, then the result corresponds to the classic result of Shannon, so we may assume that $d \geq 3$. 

First observe that it is enough to prove the result for regular graphs.
Indeed, if $G$ is not $\Delta$-regular, we can build a $\Delta$-regular graph $H$ containing $G$ as a subgraph  as follows: take two copies of $G$, and for each vertex $v$ of $G$, add $\Delta -d(v)$ edges between the two copies of $v$. Now, if $H$ is $(\lceil \frac{3\Delta - 1}{3d - 1} \rceil, d)$-edge colourable, then $G$ is too. 
Moreover, by Lemma~\ref{lem:chi_increasing}, it is enough to prove it for values of $\Delta$ such that $\lceil \frac{3\Delta - 1}{3d - 1} \rceil < \lceil \frac{3(\Delta +1) - 1}{3d - 1} \rceil$. We call such integers \textit{special}.

Let $G$ be a counterexample that minimizes  $\Delta$ and has minimum size. That is, $\Delta$ is special, $G$ is $\Delta$-regular, $\chi'_d(G) = \lceil \frac{3\Delta - 1}{3d - 1} \rceil +1$, every $\Delta$-regular graph with less vertices than $G$ is $(\lceil \frac{3\Delta - 1}{3d - 1} \rceil, d)$-edge colourable, and for every special integer $\Delta' < \Delta$, every $\Delta'$-regular graph is $( \lceil \frac{3\Delta' - 1}{3d - 1} \rceil, d)$-edge colourable. 
The result is trivial when $\Delta \leq 2$, so we can also assume $\Delta \geq 3$. 
\medskip

\begin{claim}\label{lem:extremity_shannon}
If $G$ has a bridge $e$, then a connected component of $G - e$ is isomorphic to $Sh(\Delta)$. 
\end{claim}

\begin{proofclaim}
Set $e = ab$ and let $A$ and $B$ be the two connected components of $G - e$ containing $a$ and $b$ respectively. Assume for contradiction that neither $A$ nor $B$ is isomorphic to $Sh(\Delta)$. 
Vertices of $A$ have degree $\Delta$ in $A$ except for $a$ that has degree $\Delta -1$.  Hence, we cannot have $|V(A)| \in \{1, 2\}$, nor $|V(A)| = 3$ as otherwise $A$ is isomorphic to $Sh(\Delta)$. 
We can thus assume $|V(A)| \geq 4$.    

Let $G_A$ be the graph obtained from $G$ by replacing $A$ by $Sh(\Delta)$ as in  Figure~\ref{fig:bridge}. $G_A$ is $\Delta$-regular and has strictly less vertices then $G$. Hence, by minimality of $G$, $G_A$ admits an edge colouring $c_A$ with defect $d$ using at most  $\lceil \frac{3\Delta - 1}{3d - 1} \rceil$ colours. 
We define symmetrically $G_B$ and $c_B$. We may assume, by properly permuting colours in $G_B$, that $c_B(e) = c_A(e)$. 
 We can now obtain an edge colouring of $G$ with defect $d$ using at most $\lceil \frac{3\Delta - 1}{3d - 1} \rceil$ colours by assigning colour $c_A(e)$ to $e$, colour $c_{A}(e')$ to any edge $e'$ in $B$ and colour $c_{B}(e')$ to any edge $e'$ in $A$. 
\end{proofclaim}

    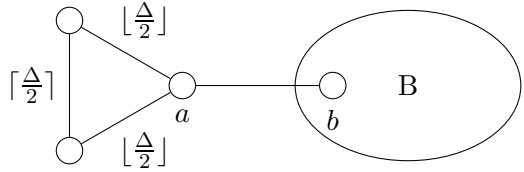
\begin{figure}[H]
    \centering
    \begin{tikzpicture}
        \vertex[label=below:$a$] (1) at (1,0) {};
        \vertex (2) at (-0.5,-0.866) {};
        \vertex (3) at (-0.5,0.866) {};
        \node (4) at (0.5,-0.866) {$\lfloor \frac{\Delta}{2} \rfloor$};
        \node (5) at (0.5,0.866) {$\lfloor \frac{\Delta}{2} \rfloor$};
        \node (6) at (-1,0) {$\lceil \frac{\Delta}{2} \rceil$};
        \node (8) at (4,0) {B};
        \draw [rotate around={90:(8)}] (8) ellipse (1cm and 1.5cm);
        \vertex[label=below:$b$] (7) at (3,0) {};
        \draw (1) -- (2);
        \draw (2) -- (3);
        \draw (3) -- (1);
        \draw (7) -- (1);
    \end{tikzpicture}
    \caption{the graph $G_A$} \label{fig:bridge}
    \end{figure}
    
Observe that, if a $\Delta$-regular graph has a bridge, then $\Delta$ must be odd. Moreover, if $\Delta$ is odd, for every $(\frac{3\Delta-1}{3d-1}, d)$ edge colouring  of $Sh(\Delta)$,  there is a colour $c$ such that the (unique) vertex of $Sh(\Delta)$ with degree $\Delta -1$ is incident with at most $d-1$ edges coloured with $c$. This simple observation is used in the proof of the following claim.

\begin{claim}\label{lem:at_most_one_bridge}
$G$ has at most one bridge.
\end{claim}

\begin{proofclaim}
Suppose for contradiction that $G$ has two bridges $uv$ and $u'v'$. By \eqref{lem:extremity_shannon}, we may assume that $G$ is made of two disjoint copies of $Sh(\Delta)$ plus a graph $A$ as in Figure~\ref{fig:two_bridges} (note that $u=u'$ is possible. See Figure~\ref{fig:two_bridges2}). 

Assume first that $u \neq u'$. Then $A + uu'$ is $\Delta$-regular and has strictly less vertices than $G$. So $A + uu'$ admits an edge colouring $c_A$ with defect $d$ using at most $\lceil \frac{3\Delta - 1}{3d - 1} \rceil$ colours.  We can extend this colouring to $G$ by giving colour $c_A(uu')$ to $uv$ and $u'v'$ and then extend this colouring to the two copies of $Sh(\Delta)$ (this is possible by the observation stated right before the claim).  

    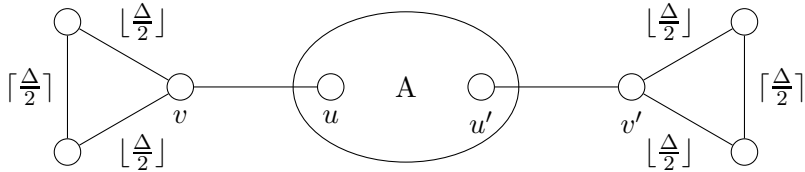
\begin{figure}[H]
    \centering
    \begin{tikzpicture}
        \begin{scope}
        \vertex[label=below:$v$] (1) at (1,0) {};
        \vertex (2) at (-0.5,-0.866) {};
        \vertex (3) at (-0.5,0.866) {};
        \node (4) at (0.5,-0.866) {$\lfloor \frac{\Delta}{2} \rfloor$};
        \node (5) at (0.5,0.866) {$\lfloor \frac{\Delta}{2} \rfloor$};
        \node (6) at (-1,0) {$\lceil \frac{\Delta}{2} \rceil$};
        \node (8) at (4,0) {A};
        \draw [rotate around={90:(8)}] (8) ellipse (1cm and 1.5cm);
        \vertex[label=below:$u$] (7) at (3,0) {};
        
        \vertex[label=below:$v'$] (9) at (7,0) {};
        \vertex (10) at (8.5,-0.866) {};
        \vertex (11) at (8.5,0.866) {};
        \node (12) at (7.5,-0.866) {$\lfloor \frac{\Delta}{2} \rfloor$};
        \node (12) at (7.5,0.866) {$\lfloor \frac{\Delta}{2} \rfloor$};
        \node (13) at (9,0) {$\lceil \frac{\Delta}{2} \rceil$};
        \vertex[label=below:$u'$] (14) at (5,0) {};
        
        \draw (9) -- (10);
        \draw (10) -- (11);
        \draw (11) -- (9);
        \draw (14) -- (9);

        \draw (1) -- (2);
        \draw (2) -- (3);
        \draw (3) -- (1);
        \draw (7) -- (1);
        \end{scope}

    \end{tikzpicture}
    \caption{the graph $G$ when $u \neq u'$}
    \label{fig:two_bridges}
    \end{figure}

Assume now that $u = u'$. We consider the graph $G'$ obtained by replacing the two copies of $Sh(\Delta)$ by four new vertices as in Figure~\ref{fig:two_bridges2}. 
As $G'$ is $\Delta$-regular and has two vertices less then $G$,  it is $(\lceil \frac{3 \Delta - 1}{3d - 1} \rceil,d)$-colourable. The obtained colouring of $A$ can easily be extended to the two copies of $Sh(\Delta)$ without any new colour (this is again possible by the observation  stated right before the claim).
\end{proofclaim}

    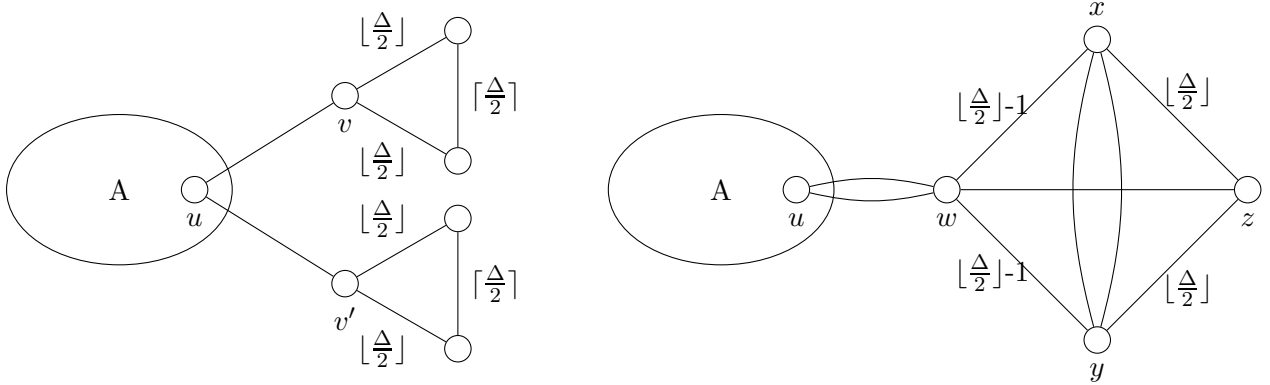
\begin{figure}[H]
    \centering
    \begin{tikzpicture}
        \begin{scope}
        \vertex[label=below:$v$] (1) at (7,1.25) {};
        \vertex (2) at (8.5,0.383) {};
        \vertex (3) at (8.5,2.116) {};
        \node (4) at (7.5,2.116) {$\lfloor \frac{\Delta}{2} \rfloor$};
        \node (5) at (7.5,0.383) {$\lfloor \frac{\Delta}{2} \rfloor$};
        \node (6) at (9,1.25) {$\lceil \frac{\Delta}{2} \rceil$};

        \node (8) at (4,0) {A};
        \draw [rotate around={90:(8)}] (8) ellipse (1cm and 1.5cm);
        \vertex[label=below:$u$] (7) at (5,0) {};
        
        \vertex[label=below:$v'$] (9) at (7,-1.25) {};
        \vertex (10) at (8.5,-2.116) {};
        \vertex (11) at (8.5,-0.383) {};
        \node (12) at (7.5,-2.116) {$\lfloor \frac{\Delta}{2} \rfloor$};
        \node (12) at (7.5,-0.383) {$\lfloor \frac{\Delta}{2} \rfloor$};
        \node (13) at (9,-1.25) {$\lceil \frac{\Delta}{2} \rceil$};
        
        \draw (9) -- (10);
        \draw (10) -- (11);
        \draw (11) -- (9);
        \draw (7) -- (9);

        \draw (1) -- (2);
        \draw (2) -- (3);
        \draw (3) -- (1);
        \draw (1) -- (7);
        \end{scope}

        \begin{scope}[xshift = 8cm]
        
        \vertex[label=below:$u$] (1) at (5,0) {};
        \node (2) at (4,0) {A};
        \draw [rotate around={90:(2)}] (2) ellipse (1cm and 1.5cm);
        \vertex[label=below:$w$] (3) at (7,0) {};
        \vertex[label=above:$x$] (4) at (9,2) {};
        \vertex[label=below:$y$] (5) at (9,-2) {};
        \vertex[label=below:$z$] (6) at (11,0) {};
        
        \node[label=above:$\lfloor \frac{\Delta}{2} \rfloor$-$1$] () at ($(3)!0.3!(4)$) {};
        \node[label=below:$\lfloor \frac{\Delta}{2} \rfloor$-$1$] () at ($(3)!0.3!(5)$) {};
        \node[label=above:$\lfloor \frac{\Delta}{2} \rfloor$] () at ($(6)!0.4!(4)$) {};
        \node[label=below:$\lfloor \frac{\Delta}{2} \rfloor$] () at ($(6)!0.4!(5)$) {};

        \draw (1) to[bend left = 13] (3);
        \draw (1) to[bend right = 13] (3);
        
        \draw (3) -- (4);
        \draw (3) -- (5);
        \draw (3) -- (6);
        \draw (4) -- (6);
        \draw (5) -- (6);
        \draw (4) to[bend left = 15] (5);
        \draw (4) to[bend right = 15] (5);
        
        \end{scope}
    \end{tikzpicture}
    \caption{On the right: the graph $G$ when $u=u'$, on the left: the graph $G'$.} \label{fig:two_bridges2}
    \end{figure}

\begin{claim}\label{lem:delta_odd_k_even_extract_factor}
$G$ has a $k$-factor for every even integer $k \leq \frac{2\Delta}{3}$. 
\end{claim}

\begin{proofclaim}
Let $k \leq \frac{2 \Delta}{3}$ be an even integer. 
If $\Delta$ is even, the result holds by Theorem~\ref{thm:petersen}. So we may assume that $\Delta$ is odd. 
If $G$ is $2$-edge connected, then we are done by Theorem~\ref{lem:factor2/3}. 
So assume $G$ has a bridge $uv$. Let $A, B$ be the two connected components of $G \setminus uv$ with $u \in V(A)$ and $v \in V(B)$. By (\ref{lem:at_most_one_bridge}), $G$ has no other bridges and thus $A$ and $B$ are both $2$-edge-connected. 
By (\ref{lem:extremity_shannon}), one of $A$ or $B$ is isomorphic to $Sh(\Delta)$.  Without loss of generality, we suppose that it is $B$. 
Let $w$ and $x$ be the two other vertices of $B$. Let $y$ be a neighbour of $u$ in $A$. 
Consider $G' = G + uv + yw - uy - vw$ (see Figure~\ref{fig:yeah}). It is easy to check that $G'$ is $\Delta$-regular and $2$-edge-connected. Applying  Theorem~\ref{lem:factor2/3} on $G'$ with $e = wy$, $G'$ has a $k$-factor $F$ containing the edge $wy$. 
There exists an integer $s \leq k-1$ such that $F$ contains $s$ edges $wx$, and $k-s-1$ edges $wv$. So $F$ must contain $k-s$ edges $vx$ and thus $F$ contains exactly one edge $uv$. 
Hence, $F - uv - yw + uy + vw$ is a $k$-factor of $G$.
\end{proofclaim}

    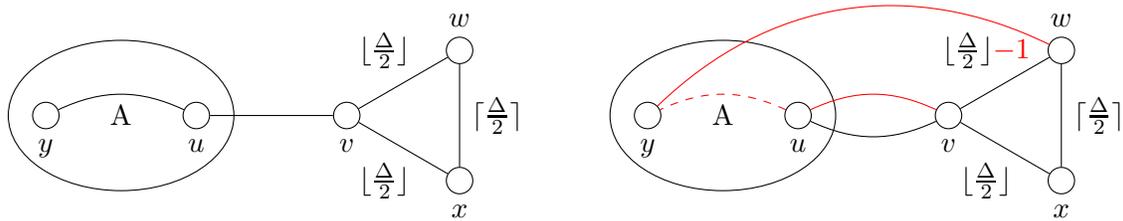
\begin{figure}[H]
    \centering
    \begin{tikzpicture}
        \begin{scope}[xshift = 0cm]
        \vertex[label=below:$v$] (1) at (-1,0) {};
        \vertex[label=below:$x$] (2) at (0.5,-0.866) {};
        \vertex[label=above:$w$] (3) at (0.5,0.866) {};
        \node (4) at (-0.5,-0.866) {$\lfloor \frac{\Delta}{2} \rfloor$};
        \node (5) at (-0.5,0.866) {$\lfloor \frac{\Delta}{2} \rfloor$};
        \node (6) at (1,0) {$\lceil \frac{\Delta}{2} \rceil$};
        \node (8) at (-4,0) {A};
        \draw [rotate around={90:(8)}] (8) ellipse (1cm and 1.5cm);
        \vertex[label=below:$u$] (7) at (-3,0) {};
        \vertex[label=below:$y$] (9) at (-5,0) {};
        \draw (1) -- (2);
        \draw (2) -- (3);
        \draw (3) -- (1);
        \draw (7) -- (1);
        \draw (7) to[bend right = 25] (9);
        \end{scope}
        
        \begin{scope}[xshift = 8cm]
        \vertex[label=below:$v$] (1) at (-1,0) {};
        \vertex[label=below:$x$] (2) at (0.5,-0.866) {};
        \vertex[label=above:$w$] (3) at (0.5,0.866) {};
        \node (4) at (-0.5,-0.866) {$\lfloor \frac{\Delta}{2} \rfloor$};
        \node (5) at (-0.5,0.866) {$\lfloor \frac{\Delta}{2} \rfloor \textcolor{red}{ - 1}$};
        \node (6) at (1,0) {$\lceil \frac{\Delta}{2} \rceil$};
        \node (8) at (-4,0) {A};
        \draw [rotate around={90:(8)}] (8) ellipse (1cm and 1.5cm);
        \vertex[label=below:$u$] (7) at (-3,0) {};
        \vertex[label=below:$y$] (9) at (-5,0) {};
        \draw (1) -- (2);
        \draw (2) -- (3);
        \draw (3) -- (1);
        \draw (7) to[bend right = 25] (1);
        \draw[red] (7) to[bend left = 25] (1);
        \draw[red] (9) to[bend left = 35] (3);
        \draw[dashed,red] (7) to[bend right = 25] (9);
        \end{scope}
        
    \end{tikzpicture}
    \caption{The graphs $G$ and $G'$} \label{fig:yeah}
    \end{figure}
    
Having established the three claims (1), (2) and (3), we can proceed to prove the theorem.
We begin by treating the five smallest special values of $\Delta$.   
\smallskip

\textbf{Case 1:} $\lceil \frac{3\Delta - 1}{3d - 1} \rceil = 1$, $\Delta = d$. The result holds trivially. 
\smallskip

\textbf{Case 2:}\label{case:delta_two_d_minus_one} 
$\lceil \frac{3\Delta - 1}{3d - 1} \rceil = 2$, $\Delta = 2d - 1$. 
We can apply (\ref{lem:delta_odd_k_even_extract_factor}) with $k = d - 1$ to get a $(d-1)$-factor of $G$, whose complementary in $G$ is a $d$-factor. Thus $\chi'_{d}(G) \leq 2$. This proves case 2. 
\smallskip

\textbf{Case 3:}\label{case:delta_three_d_minus_one} 
$\lceil \frac{3\Delta - 1}{3d - 1} \rceil = 3$, $\Delta = 3d - 1$. Observe that $\Delta$ is even. 
By Theorem~\ref{thm:petersen}, $G$ has a $2d$-factor $F$. 
By applying Lemma~\ref{lem:delta_even_kV_even} on connected components of even size of $F$ and Lemma~\ref{lem:delta_even_kV_odd} on connected components of odd size, we can extract  graphs $G_{A}$ and $G_{B}$ along with a matching $M$ such that $E(F) = E(G_A) \cup E(G_B) \cup M$, $\Delta(G_{A}) \leq d$, $\Delta(G_{B}) \leq d$. Now, $E(G)$ can be partitioned into $E(G_A)$, $E(G_B)$ and $E(G) \cup M \setminus F$, each of these sets  having maximum degree at most $d$. This proves case 3. 
\smallskip

\textbf{Case 4:} 
$\lceil \frac{3\Delta - 1}{3d - 1} \rceil = 4$, $\Delta = 4d - 1$. 
By applying (\ref{lem:delta_odd_k_even_extract_factor}) with $k = 2d$, we get a $2d$-factor $A$ of $G$, whose complementary in $G$ is a $(2d-1)$-factor $B$. By applying Lemma~\ref{lem:delta_even_kV_even} on components of $A$ of even size and Lemma~\ref{lem:delta_even_kV_odd} on components of $A$ of odd size, we get $A = A_{1} \cup A_{2} \cup M$ where $\Delta(A_{1}) \leq d$, $\Delta(A_{2}) \leq d$ and $M$ is a matching.

It remains to prove that $\chi'_{d}(B \cup M) \leq 2$. Let $C$ be a connected component of $B \cup M$. 
If every vertex of $C$ is incident with an edge of $M$, then $C$ has an even number of vertices and is $2d$-regular, and so $\chi'_d(C) = 2$ by Lemma~\ref{lem:delta_even_kV_even}. 
Assume now that there exists a vertex of $C$ that is not incident with an edge of $M$. Take two copies of $C$, and add an edge between the copies of each vertex of $C$ not incident with an edge of $M$. The obtained graph has an even number of vertices and is $2d$-regular, so it is $(2,d)$-edge colourable by Lemma~\ref{lem:delta_even_kV_even} and thus so is $C$. So each connected component of $B \cup M$ is $(2,d)$-edge colourable, and thus so is $B \cup M$. This proves case 4.
\smallskip

\textbf{Case 5:} 
$\lceil \frac{3\Delta - 1}{3d - 1} \rceil = 5$, $\Delta = 5d - 2$. 
By applying~(\ref{lem:delta_odd_k_even_extract_factor}) with $k = 3d - 1$, we can partition $G$ into a $(2d-1)$-factor $A$ and a $(3d-1)$-factor $B$. By respectively using cases 2 and 3, we have $\chi'_d(A) \leq 2$ and $\chi'_d(B)\leq 3$, so $\chi'_d(G) \leq 5$. This proves case 5. 
\medskip

We may now assume that $\Delta \geq 5d-1$. 
If $\Delta$ is even, then $G$ has a $(3d-1)$-factor by Petersen Theorem~(\ref{thm:petersen}), and if $\Delta$ is odd, since $3d-1 \leq \frac{2(5d-1)}{3} = \frac{2\Delta}{3}$ and $3d-1$ is even, $G$ admits a $(3d-1)$-factor by \eqref{lem:delta_odd_k_even_extract_factor}. 
Let $F$ be a $(3d-1)$-factor of $G$. 
By Case 3, $F$ is $(3,d)$-edge colourable. As $G - F$ is $(\Delta - (3d - 1))$-regular, by minimality of $\Delta$ we have:
$$\chi'_{d}(G-F) \leq \lceil \frac{3(\Delta - (3d - 1)) - 1}{3d - 1} \rceil,$$
and thus 
$$\chi'_{d}(G) \leq 3 + \lceil \frac{3(\Delta - (3d - 1)) - 1}{3d - 1} \rceil = \lceil \frac{3\Delta - 9d + 3 - 1 + 9d - 3}{3d-1} \rceil = \lceil \frac{3\Delta - 1}{3d - 1} \rceil.$$

\end{proof}


\section{Generalization of Vizing Theorem}\label{sec:simple}


In this section, we will only consider simple graphs. Vizing~\cite{V64} proved the following theorem :

\begin{theorem}[Vizing's Theorem,  \cite{V64}] \label{thm:vizing}
For every simple graph $G$ with maximum degree $\Delta$, $\chi'_{1}(G) \in \{ \Delta, \Delta + 1 \}$.
\end{theorem}

While there are only $2$ possibilities, deciding between them was proven to be NP-complete even for regular simple graphs. 

\begin{theorem}[ Holyer~\cite{H81}, Leven and Galil~\cite{LZ83}]
\label{thm:chromatic_index_NP_complete}
For every $\Delta \geq 3$, it is NP-complete to decide if a $\Delta$-regular simple graph $G$ is $\Delta$-edge colourable. 
\end{theorem}

Vizing's theorem easily implies its following generalization to $d$-defective edge colouring. 
\begin{corollary}\label{coro:vizing}
For every $d \geq 1$ and every simple graph $G$ with maximum degree $\Delta$, $\chi'_{d}(G) \in \{ \lceil \frac{\Delta}{d} \rceil, \lceil \frac{\Delta+1}{d} \rceil \}$.
\end{corollary}

\begin{proof}
The lower bound holds by Lemma~\ref{lem:trivialbound}. For the upper bound,  consider an edge colouring of $G$ with $\Delta(G) + 1$ colours (it exists by Vizing's Theorem) and let $M_1 , \dots , M_{\Delta(G) + 1}$ be the classes of colours. By assigning colour $1$ to $M_1 \cup \dots \cup M_{d}$, colour $2$ to $M_{d + 1} \cup \dots \cup M_{2d}$,  etc,  we obtain a $(\lceil \frac{\Delta+1}{d} \rceil,d)$ edge colouring of $G$.
\end{proof}

We now prove a generalization of Theorem~\ref{thm:chromatic_index_NP_complete} in the context of defective edge colouring. Before that, we need the following construction.  

For every integer $k,d \geq 1$, we construct a simple graph $G_{kd, d}$ such that $G$ is $kd$-regular and $\chi'_{d}(G) = k$. 
We can set $G_{d,d} = K_{d+1}$. Inductively, having defined $G_{kd,d}$, let $G_{(k+1)d,d}$ be the simple graph obtained by taking two disjoint copies of $G_{kd,d}$  and adding the edges of any $d$-regular bipartite simple graph between these two copies\footnote{For example, naming $u_1, \dots, u_n$ and $v_1, \dots, v_n$ the vertices of the two copies of $G_{kd,d}$, add the edges $u_iv_i, u_iv_{i+1}, \dots, u_iv_{i+d}$ for $i=1, \dots, n$, subscripts being taken modulo $n$. It gives a $d$-regular bipartite simple graph as soon as $n \geq d$.}. 
The obtained simple graph is clearly $(k+1)d$-regular, and we can $(k+1, d)$-edge colour it by taking a $(k,d)$-edge colouring for the two copies of $G_{kd,d}$ and add a new colour for the added edges, and finally  by Lemma~\ref{lem:trivialbound} it does not admit a $(k,d)$-edge colouring. Hence  $\chi'_d(G_{(k+1)d,d}) = k+1$.

\begin{theorem}\label{thm:NP}
Let $d, \Delta \geq 1$, and $G$ a simple graph with maximum degree $\Delta$.  
Then $\chi'_d(G) = \lceil \frac{\Delta}{d} \rceil$ if:
\begin{itemize}
    \item $d$ does not divide $\Delta$ or, 
    \item $d$ is even or, 
    \item $\Delta = d$ or,
    \item $\Delta = 2d$ and every $2d$-regular connected component of $G$ has an even number of vertices. 
\end{itemize}
Moreover, if $d$ is odd, $\Delta = 2d$ and a $2d$-regular connected component of $G$ has an odd number of vertices, then $\chi'_{d}(G) = 3$. 
Finally, in every other case, that is if $d$ is odd and $\Delta = kd$ for some $k \geq 3$, it is NP-complete to decide if $\chi'_d(G) = k $. 
\end{theorem}

\begin{proof}
The first case is a direct consequence of Corollary~\ref{coro:vizing}, noticing that if $d$ does not divide $\Delta$, then $\lceil \frac{\Delta(G)}{d} \rceil = \lceil \frac{\Delta(G)+1}{d} \rceil$. 
The second case has already been proven. See Theorem~$\ref{thm:deven}$.
The third case is trivial. 
To prove the fourth case, take two disjoint copies of $G$, and for each vertex $v$ of $G$, add $2d - d(v)$ edges between the two copies of $v$. The obtained (not necessarily simple) graph $G'$ is $2d$-regular, and each of its connected component has even size  (as the connected components of odd size of $G$ were not regular and thus are included in connected components of even size of $G'$). Now, by Lemma~\ref{lem:delta_even_kV_even}, $G'$ is $(2,d)$-edge colourable, and so is $G$.

Assume now that $d$ is odd, $\Delta = 2d$, and $G$ has a $2d$-regular connected component $C$ of odd size. As $C$ is of odd size and $d$ is odd, $G$ does not admit a $d$-factor, and thus cannot be $(2,d)$ edge coloured. So, by Corollary~\ref{coro:vizing}, $\chi'_d(G) = 3$. 

Finally, assume that $d$ is odd and $\Delta = kd$ for some integer $k \geq 3$. 
Deciding if $\chi'_d(G) = k$ is clearly in NP. 
 We perform a reduction from the case $d=1$ (which is NP-complete by Theorem~\ref{thm:chromatic_index_NP_complete}).  Let $G$ be a $k$-regular simple graph.

We construct a simple graph $G'$ as follows: start with a copy of $G$, then for each vertex $v$ of $G$, add $\frac{k(d-1)}{2}$ disjoint copies of $G_{kd,k}$, remove one edge $ab$ and add edges $av$ and $bv$ for each copy. 
The graph $G'$ is clearly simple, and $kd$-regular. 

We now prove that $\chi'_{1}(G) = k$ if and only if $\chi'_{d}(G') = k$. 
Assume  first that there exists a $(k,1)$ edge colouring of $G$ using $k$ colours and let us show how to find a $(k,d)$-edge colouring of $G'$, which would imply that $\chi'_d(G) = k$ by Corollary~\ref{coro:vizing}. We use colours from $\{1, 2,\dots, k\}$.  
Start with a $(k,1)$ edge colouring of the copy of $G$ in $G'$. Observe that each vertex is incident with exactly one edge of each colour. 
For every vertex $u$ in the copy of $G$, and for $i=1,2, \dots, \frac{k(d-1)}{2}$, assign colour $1 + (i \mod k)$ to the two edges linking $u$ and the $i^{th}$ copy $G_{kd,k}$ (after we have removed an edge). Now, each vertex $u$ in the copy of $G$ is incident with precisely $d$ edges of each colour. It remains to extend the colouring to the copies of $G_{kd,d}$ (after we have removed an edge) which can easily be done since $\chi'_d(G_{kd,d}) = k$ and for each copy, naming $ab$ the missing edge, each of $a$ and $b$ is incident with a single coloured edge, both of the same colour.  

We now assume that $\chi'_d(G') = k$. We are going to show that in any $(k,d)$-edge colouring of $G'$, the copy of $G$ in $G'$ is $(k,1)$-edge coloured, implying that $\chi'_1(G) = k$ by Vizing Theorem. So let us start with a $(k,d)$ edge colouring of $G'$. 
Let $u$ be a vertex in the copy of $G$ in $G'$, and let $C$ be one of the $\frac{k(d-1)}{2}$ copies of $G_{kd,d}$ (from which we have removed an edge) linked with $u$. 
We name $a$ and $b$ the vertices in $C$ incident with the missing edge and observe that in $C$, each vertex has degree $kd$ except for $a$ and $b$ that have degree $kd-1$.  
Since $|E(C)| = \frac{kd|V(C)|}{2} - 1$ and each colour class covers at most $\frac{d|V(C)|}{2}$ edges of $C$,  
it must be that each colour class covers $\frac{d|V(C)|}{2}$ edges of $C$, that is each colour class induces a $d$-regular simple graph, except for one colour class that covers only $\frac{d|V(C)|}{2} -1$ edges. 
Naming $c$ this colour, we get that vertices $a$ and $b$ are incident with exactly $d-1$ edges coloured $c$, and $d$ edges of every other colour. So the edges $au$ and $bv$ must receive the colour $c$. 

In particular, for each copy of $G_{kd,d}$ linked with $u$, the two edges linking $u$ with this copy must receive the same colour.  
Moreover, each colour must appear between $u$ and precisely $\frac{d-1}{2}$ copies of $G_{kd,d}$, for otherwise $u$ would be incident with more than $d$ edges of a given colour.
Hence, each colour appears $d-1$ times in the $k(d-1)$ edges linking $u$ with the copies of $G_{kd,d}$ and thus each colour appears exactly once in the $k$ edges incident with $u$ in the copy of $G$. 
Hence, as announced, the $(k,d)$-edge colouring of $G'$ induces a $(k,1)$-edge colouring of $G$. 
\end{proof}

We point out that Vizing\cite{V64} also proved that for every (not necessarily simple) graph $G$ with maximum degree $\Delta$ and edge multiplicity $\mu$, $\chi'_{1}(G) \leq \Delta + \mu$ where the edge multiplicity is the maximum number of edges between two vertices. This directly implies that $\chi'_{d}(G) \leq \lceil \frac{\Delta + \mu}{d} \rceil$.


\section{Further works}\label{sec:further}

\subsection*{List colouring}
The \textit{$d$-defective list chromatic index} of a graph $G$, denoted by $ch'_d(G)$, is defined as the minimum $k$ such that, for any choice of list of $k$ integers given to each edge, there is an edge colouring with defect $d$ such that each edge receives a colour from its list.  So $ch'_1(G)$ is the usual list chromatic index. 

Borodin et al.~\cite{10.1006/jctb.1997.1780} proved that Shannon bound holds for the list chromatic index, that is, for every graph $G$, 
$ch_1'(G) \leq \lceil \frac{3 \Delta(G)}{2} \rceil$. It is then natural to ask if Theorem~\ref{thm:main_theorem} extends to defective list edge colouring. 
As mention in the introduction, when $d$ is even, it is proved in~\cite{HILTON2001253} (and a simpler proof is given in~\cite{amini:inria-00144318}) that for every graph $G$,  $ch'_d(G) = \lceil \frac{\Delta(G)}{d} \rceil$. 
When $d$ is odd, a proof that $chi'_{d}(G) \leq \lceil \frac{3\Delta}{3d - 1} \rceil$ is announced in~\cite{amini:inria-00144318} but seems to have a flaw and actually holds only in the case where $\Delta$ is divisible by $3k-1$. 

\begin{conjecture}
For every odd integer $d$ and for every  graph $G$, $ch'_d(G) \leq  \lceil \frac{3\Delta - 1}{3d-1} \rceil$
\end{conjecture}

We finally mention the following stronger conjecture that corresponds to the infamous list edge colouring conjecture for $d=1$ and is proved for bipartite graph in~\cite{HILTON2001253}. 

\begin{conjecture}\cite{HSS98}
For every graph $G$ and every integer $d$, $ch'_d(G) = \chi'_d(G)$.
\end{conjecture}

\subsection*{The Goldberg-Seymour Conjecture}
Let $d \in \mathbb N^*$ and $G$ a graph. Observe that in any edge-coloring of $G$ with defect $d$, and for any $X \subseteq V(G)$, each color class contains at most $\lfloor \frac{d|X|}{2} \rfloor$ edges, which leads to the following lower bound on the $d$-defective edge chromatic number of any graph $G$: 

$$
    \chi'_d(G) \leq \Gamma_d(G) = 
    \max 
    \Big\{ \Bigl\lceil \frac{|E(G[X])|}{\lfloor \frac{d|X|}{2}\rfloor} \Bigr\rceil \mid X \subseteq V(G) \Big\}.
$$
The following was known as the Goldberg-Seymour Conjecture~\cite{G73, S79} for almost 50 years. Recently, Chen et al.~\cite{CJZ19} announced a proof. 

\begin{conjecture}[Golberg-Seymour~\cite{G73, S79}]\label{conj:goldberg}  
For every graph $G$,\\ $\chi'_1(G) \leq \max\{\Gamma_1(G), \Delta(G) + 1\}$. 
\end{conjecture}

We think that the following generalization could hold. 

\begin{conjecture}\label{conj:Golderg_gen}
Every graph $G$ satisfies $\chi'_d(G) \leq \max\{\Gamma_d(G), \lceil \frac{\Delta(G) +1}{d} \rceil\}$.
\end{conjecture}

An easy proof of the conjecture could start as follows.  A counter-example $G$ to the conjecture must satisfy $\Delta(G) +1 < \chi'_1(G) =\Gamma_1(G)$. 
This implies that $\chi'_d(G) \leq \lceil \frac{\Gamma_1(G)}{d} \rceil$. 
So it is enough to prove that $\lceil \frac{\Gamma_1(G)}{d} \rceil \leq \max\{\Gamma_d(G), \lceil \frac{\Delta(G) +1}{d} \rceil\}$. Unfortunately this last inequality does not hold, for example on the following simple example.  
Consider the graph $G$ made of three vertices connected by respectively $7$, $7$ and $2$ edges. The followings hold: 
$$15 = \Delta(G) +1 < \Gamma_1(G)= \chi'_1(G) = 16,$$ whereas 
$$4 =  \Gamma_3(G) < \lceil \frac{\Delta(G) + 1}{3}\rceil =  \chi'_3(G) = 5.$$

\subsection*{The degree Ramsey number of stars}

In this subsection, we briefly describe the link between the degree Ramsey number of stars and defective edge colouring. We are thankfull to Ross Kang for bringing this to our attention.    

Let $H$, $G$ be simple graphs. Let $H \rightarrow_s G$ means that every colouring of $E(H)$ with $s$ colours produces a monochromatic copy of $H$. 
The \textit{degree Ramsey number} of a simple graph $G$ is $R_{\Delta}(G;s) = \min\{\Delta(H): H \rightarrow_s G\}$. 
Observe that $H \rightarrow_s K_{1,d+1}$ means that $\chi'_d(H) \geq s+1$. Hence,  $R_{\Delta}(K_{1,d+1};s) = \min\{\Delta(H): \chi'_d(H) \geq s+1\}$. 

It can be proved (with a little brain gymnastic) that the following result of Kinnersley, Milans and West is equivalent to corollary~\ref{coro:vizing}. 

\begin{theorem}\cite{kinnersley_milans_west_2012}
If $s \geq 2$, then 
$
R_{\Delta}(K_{1,d+1}; s) = \left\{
    \begin{array}{ll}
       s\cdot d & \mbox{if } d \mbox{ is odd} \\
        s\cdot d +1 &  \mbox{if } d \mbox{ is even}
    \end{array}
\right.
$
\end{theorem}

It could be of interest to look at the degree Ramsey number of (multi)graphs. 
\bigskip

{\bf Acknowledgment}: 
This work was supported by the group Casino/ENS Chair on Algorithmics and Machine Learning.

\end{document}